\documentclass[11pt, reqno]{amsart}
\usepackage[utf8]{inputenc}
\usepackage{amsmath}
\usepackage{amsthm}
\usepackage{amsfonts}
\usepackage{MnSymbol}
\usepackage{blindtext,amsmath,amsthm,amsfonts, textcomp, mathrsfs, mathtools}
\usepackage[shortlabels]{enumitem}

\usepackage[top=30mm,bottom=30mm,left=30mm,right=30mm]{geometry}
\newtheorem{theorem}{Theorem}[section]
\newtheorem{conjecture}{Conjecture}

\newtheorem{remark}[theorem]{Remark}
\newtheorem{question}[theorem]{Question}
\newtheorem*{theorem*}{Theorem}

\newtheorem*{problem*}{Problem}
\newtheorem*{conjecture*}{Conjecture}
\newtheorem*{question*}{Question}
\newtheorem{lemma}[theorem]{Lemma}
\newtheorem{proposition}[theorem]{Proposition}

\begin{document}
\title[Distribution and Non-vanishing of special values of $L$-series]{Distribution and Non-vanishing of special values of $L$-series attached to Erd\H{o}s functions}

\author[Siddhi Pathak]{Siddhi Pathak}

\address{Department of Mathematics and Statistics, Queen's University, Kingston, Canada, ON K7L 3N6.}
\email{siddhi@mast.queensu.ca}



\begin{abstract}
In a written correspondence with A. Livingston, Erd\H{o}s conjectured that for any arithmetical function $f$, periodic with period $q$, taking values in $\{-1,1\}$ when $q \nmid n$ and $f(n)=0$ when $q \mid n$, the series $\sum_{n=1}^{\infty} f(n)/n$ does not vanish. This conjecture is still open in the case $q \equiv 1 \bmod 4$ or when $2 \phi(q)+ 1 \leq q$. In this paper, we obtain the characteristic function of the limiting distribution of $L(k,f)$ for any positive integer $k$ and Erd\H{o}s function $f$ with the same parity as $k$. Moreover, we show that the Erd\H{o}s conjecture is true with ``probability" one.
\end{abstract}

\maketitle


\section{Introduction}\label{intro}

Inspired by Dirichlet's theorem that $L(1,\chi) \neq 0$ for a non-principal Dirichlet character $\chi$, Sarvadaman Chowla \cite{chowla} initiated the study of non-vanishing of the series $$\sum_{n=1}^{\infty} \frac{f(n)}{n}$$for any periodic arithmetical function $f$ whenever the above series converges. Since then, this question has been extensively studied by many authors in a variety of settings (see, for example \cite{tijdeman}, \cite{ram-saradha-1}, \cite{ram-kumar}, \cite{ram-tapas-me} etc.). Most of the study is concentrated around the case when $f$ is supported on the coprime residue classes modulo $q$. However, very little is known otherwise. \\

Possibly the simplest example of an investigation in this scenario is the following conjecture made by Erd\H{o}s in a written correspondence with A. Livingston \cite{livingston}. 
\begin{conjecture}\label{erdos-conj}
Let $q$ be a positive integer. Let $f$ be an arithmetical function, periodic with period $q$ such that 
\begin{equation*}
    f(n) = 
    \begin{cases}
    \pm 1 & \text{ if } q \nmid n, \\
    0 & \text{ if } q \mid n.
    \end{cases}
\end{equation*}
Then the series $\sum_{n=1}^{\infty} f(n)/n \neq 0$, whenever it converges.
\end{conjecture}

For any periodic function $f$, one can define the $L$-series
\begin{equation*}
    L(s,f) := \sum_{n=1}^{\infty} \frac{f(n)}{n^s},
\end{equation*}
which converges absolutely for $\Re(s)>1$. Using the theory of the Hurwitz zeta-function, this series can be analytically continued to the entire complex plane except for a simple pole at $s=1$ with residue $\frac{1}{q} \sum_{a=1}^q f(a)$. Thus, we see that the series $L(1,f) = \sum_{n=1}^{\infty} f(n)/n $ converges if and only if $\sum_{a=1}^q f(a)=0$. Hence, for the sake of brevity, we say that a rational valued function $f$ on the integers, periodic with period $q$ is an Erd\H{o}s function mod $q$ if $f(n) \in \{-1,1\}$ when $q \nmid n$ and $f(n) = 0$ otherwise and $\sum_{a=1}^q f(a) = 0$.\\

Erd\H{o}s functions may be viewed as non-multiplicative analogues of quadratic Dirichlet characters. For a fundamental discriminant $D$, let $\chi_D$ be the quadratic character modulo $|D|$ given by the Kronecker symbol, i.e, 
\begin{equation*}
    \chi_D (n) := \bigg(\frac{D}{n} \bigg).
\end{equation*}
In 1951, S. Chowla and P. Erd\H{o}s \cite{chowla-erdos} proved that the limit as $N$ tends to infinity of the frequencies
\begin{equation*}
    \frac{\# \bigg\{ D \, : \, |D| \leq N, \, L(1,\chi_D) \leq x \bigg\}}{N}
\end{equation*}
exists for all real $x$ and is a continuous distribution function. In the 1960s, Barban \cite{barban1, barban2} calculated moments of $L(1,\chi_D)$ for all integer orders $k > 0$ and thus, showed that the characteristic function of the corresponding distribution has the form
\begin{equation*}
    \sum_{k=0}^{\infty} \frac{r(k)}{k!} \, {(it)}^k.
\end{equation*}
Here $$r(k) = \sum_{\substack{n=1, \\ n \text{ odd}}}^{\infty} \frac{\phi(n) \, \tau_k(n^2)}{n^3},$$ where $\tau_k(n)$ is the $k^{\text{th}}$ divisor function, i.e., number of ways of writing $n$ as a product of $k$ natural numbers. This was also obtained by P. D. T. A. Elliott \cite[Theorem 22.1]{elliot} exploiting the multiplicative nature of quadratic characters. More specifically, Elliott proved that
\begin{theorem}
There is a distribution function $F(z)$ so that
\begin{equation*}
    \nu_x \bigg( D; h(-D) \leq \frac{e^z}{\pi} \sqrt{D} \bigg) = F(z) + O \bigg( \sqrt{\frac{\log \log x}{\log x}} \bigg)
\end{equation*}
holds uniformly for all real $z$ and real $x \geq 9$. $F(z)$ has a probability density, may be differentiated any number of times and has the characteristic function
\begin{equation*}
    \prod_{p} \left( \bigg( \frac{1}{p} + \frac{1}{2} {\bigg(1 - \frac{1}{p} \bigg) \bigg)}^{1-it} \, + \, \frac{1}{2} \bigg( 1 - \frac{1}{p}\bigg) {\bigg(1 + \frac{1}{p} \bigg)}^{-it}\right)
\end{equation*}
which belongs to the Lebesgue class $L(-\infty, \infty)$.
\end{theorem}
By Dirichlet's class number formula, for $D \geq 4$, 
\begin{equation*}
    L(1,\chi_{-D}) = \frac{\pi}{\sqrt{D}} \, h( -D ).
\end{equation*}
Thus, Elliott's theorem proves that the distribution function of the values $L(1,\chi_{-D})$ is smooth.
Further work in this context was carried out by A. Granville and K. Soundararajan \cite{granville-sound} where in they established the conjectures of Montgomery and Vaughan regarding the distribution of the extreme values of $L(1,\chi_D)$.\\

Pursuing the analogy of Erd\H{o}s functions as non-multiplicative analogs of quadratic Dirichlet characters, one may ask if certain properties of quadratic Dirichlet characters are also satisfied by Erd\H{o}s functions. In this paper, we follow the approach of Barban to understand the distribution of special values of $L$-series attached to Erd\H{o}s functions. For a positive integer $q$, let $E_q$ be the set of Erd\H{o}s functions mod $q$. Note that the convergence condition $\sum_{a=1}^q f(a) = 0$ implies that $E_q$ is non-empty only when $q \geq 3$ is odd. \\

Let $B_m$ denote the $m^{\text{th}}$ Bernoulli number defined by the generating function
\begin{equation*}
    \frac{z}{e^z - 1} = \sum_{m=0}^{\infty} \frac{B_m}{m!} \, z^m.
\end{equation*}
For any positive integer $n$, let $P_n$ denote the partially ordered set of partitions of $n$, i.e., for partitions $\underline{\lambda}$ and $\underline{\eta}$ of $n$, $\underline{\eta} \leq \underline{\lambda}$ if the parts of $\underline{\eta}$ can be obtained by merging the parts of $\underline{\lambda}$. The symbol ${\underline{\lambda} \choose \underline{\eta}}$ counts the number of ways in which parts of $\underline{\lambda}$ can be merged to obtain $\underline{\eta}$. For each $\underline{\lambda} \in P_n$, we inductively define $c(\underline{\lambda})$ as follows.
\begin{equation*}
    c((n)) := 2^{2nk} {((k-1)!)}^{2n} \bigg({(-1)}^{nk+1} \frac{B_{2nk}}{2nk!} \bigg)
\end{equation*}
and
\begin{equation}\label{c(lambda)}
    c(\underline{\lambda})=((\lambda_1,\cdots,\lambda_m)) : = \bigg[ 2^{2nk} {((k-1)!)}^{2n} \bigg(\prod_{i=1}^m {(-1)}^{\lambda_{i}k+1} \frac{B_{2\lambda_{i}k}}{2\lambda_{i}k!} \bigg) \bigg] - \sum_{\underline{\eta} < \underline{\lambda}} {\underline{\lambda} \choose \underline{\eta}} c( \underline{\eta}).
\end{equation}
Then, using the method of moments, we show that
\begin{theorem}\label{distribution}
Fix a positive integer $k \geq 1$. For any integer $r \geq 1$ and real $x$, let 
\begin{equation*}
    E_{2r+1}^{(k)}:= \{ f \in E_{2r+1} \, : \, f \text{ is of the same parity as } k \}
\end{equation*}
and 
\begin{equation*}
    F_r(x) := \frac{\# \big\{f \in E_{2r+1}^{(k)} \, : \, L(k,f) \leq x \big\}}{\# E_{2r+1}^{(k)}}.
\end{equation*}
Then, $F_r(x)$ converges to a distribution $F(x)$ at every point of continuity of the latter. Moreover, the corresponding characteristic function is entire and is given by
\begin{equation*}
    \phi(t) = \sum_{n=0}^{\infty} \frac{M(2n)}{(2n)!} \, {(-t)}^{n},
\end{equation*}
where 
\begin{equation*}
    M(2n) := \frac{\pi^{2nk}}{{((k-1)!)}^{2n}\, 2^{2n}} \, \bigg( \sum_{\underline{\lambda} \in P_n} c( \underline{\lambda})\bigg).
\end{equation*}
\end{theorem}

The restriction on the parity of the functions is inherent in the nature of the value $L(1,f)$ for any periodic function $f$. Indeed, if $f$ is a $q$-periodic rational valued function with $\sum_{a=1}^q f(a) = f(q) = 0$, then Gauss's formula for the digamma function (see \cite{livingston} for details) gives
\begin{equation*}
    L(1,f) = \frac{-\pi}{2q}\sum_{a=1}^{q-1} f(a) \, \cot \left( \frac{a \pi}{q}\right) + \frac{2}{q} \sum_{0 < j \leq q/2} \log \sin \frac{\pi j}{q} \sum_{a=1}^{q-1} f(a) \, \cos \bigg( \frac{2 \pi a j}{q}\bigg).
\end{equation*}
Therefore, when $f$ is odd, the second term in the above expression vanishes and $L(1,f)$ simplifies to an algebraic multiple of $\pi$. On the other hand, if $f$ is not odd, then $L(1,f)$ is a linear form in logarithms of algebraic numbers making it intractable, especially in order to use the method of moments. Another difficulty that arises in studying the nature of these values as opposed to those of Dirichlet $L$-functions is the absence of multiplicativity among the coefficients. Thus, the probabilistic method of Elliott does not apply to this scenario.\\

Although some progress has been made towards Conjecture \ref{erdos-conj}, it remains open in the cases $q \equiv 1 \bmod 4$ or $q > 2 \phi(q) + 1$. Conjecture \ref{erdos-conj} follows from a theorem of Baker, Birch and Wirsing \cite{bbw} when $q$ is prime. It was proved for $q < 2 \phi(q) + 1$ by T. Okada \cite{okada} and for $q \equiv 3 \bmod 4$ by M. Ram Murty and N. Saradha \cite{ram-saradha}. In 2015, T. Chatterjee and M. Ram Murty \cite{ram-tapas} approached Conjecture \ref{erdos-conj} from a density theoretic perspective. They showed that if $$S(x) := \# \big\{ q\equiv 1 \bmod 4, \, q \leq x \, : \, \text{Erd\H{o}s's conjecture holds for } q \big\}, $$then
\begin{equation*}
    \lim_{x \rightarrow \infty} \frac{S(x)}{x/4} \geq 0.82.
\end{equation*}
This can be interpreted as the Erd\H{o}s conjecture being true for at least 82\% of $q \equiv 1 \bmod 4$, which is the best possible lower bound using their methods. In this paper, we use an alternate approach and improve on their result by proving that
\begin{theorem}\label{density-1}
Let $q \geq 3$ be an odd positive integer. Let 
$E_q$ be the set of Erd\H{o}s functions mod $q$ and $V_q := \big\{f \in E_q \, : \, L(1,f) = 0 \big\}$.
Then,
\begin{equation*}\label{L(1,f)-density-0}
    \lim_{x \rightarrow \infty} \bigg\{ \bigg( {\sum_{\substack{3 \leq q \leq x, \\ q \text{ odd }}} \# V_q} \bigg) \bigg/ \bigg( {\sum_{\substack{3 \leq q \leq x,\\ q \text{ odd }}} \# E_q } \bigg) \bigg\} = 0.
\end{equation*}
\end{theorem}
This shows that Conjecture \ref{erdos-conj} is true with ``probability" one. \\

Another aspect of this question is the non-vanishing of special values of $L(s,f)$ for an Erd\H{o}s function $f$ at positive integers greater than $1$. In this direction, we ask the following question.

\begin{question}
Let $q>2$ and $k>1$ be integers and $f$ be an Erd\H{o}s  function mod $q$. Then is it true that $L(k,f) \neq 0$?
\end{question}

\begin{remark}
Let $f$ be an Erd\H{o}s function mod $q$ and $k>1$ be an integer. Suppose that $L(k,f)=0$. Then observe that
\begin{equation*}
    |f(1)| = \bigg| \sum_{n=2}^{\infty} \frac{f(n)}{n^k} \bigg| \leq \zeta(k) - 1.
\end{equation*}
This implies that $2 \leq \zeta(k)$, i.e., $k < 2$. This establishes that $L(k,f) \neq 0$ for any Erd\H{o}s function $f$ if $k \geq 2$. Thus, Erd\H{o}s's conjecture is interesting when $k = 1$.
\end{remark}



\section{Preliminaries}\label{prelim}

In this section, we introduce the results to be used later. 

\subsection{$L$-series attached to periodic functions}
Let $q$ be a fixed positive integer and $\overline{\mathbb{Q}}$ denote the algebraic closure of $\mathbb{Q}$. Consider $f : \mathbb{Z} \to \overline{\mathbb{Q}}$, periodic with period $q$. Define
\begin{equation*}
L(s,f) = \sum_{n=1}^{\infty} \frac{f(n)}{n^s}.
\end{equation*}
Observe that $L(s,f)$ converges absolutely for $\Re(s) > 1$. Since $f$ is periodic,
\begin{equation*}\label{hurw-form}
\begin{split}
L(s,f) & = \sum_{a=1}^q f(a)\sum_{k=0}^{\infty} \frac{1}{{(a + kq)}^s}\\
& = \frac{1}{q^s} \sum_{a=1}^q f(a) \zeta(s, a/q),
\end{split}
\end{equation*}
where $\zeta(s,x)$ is the Hurwitz zeta-function. For $\Re(s) > 1$ and $0 < x \leq 1$, the Hurwitz zeta-function is defined as
\begin{equation*}
\zeta(s,x) = \sum_{n=0}^{\infty} \frac{1}{{(n+x)}^s}.
\end{equation*}
In 1882, Hurwitz \cite{hurwitz} obtained the analytic continuation and functional equation of $\zeta(s,x)$. He proved:
\begin{theorem}
The Hurwitz zeta-function, $\zeta(s,x)$ extends analytically to the entire complex plane except for a simple pole at $s=1$ with residue $1$. In particular, 
\begin{equation}\label{hurwitz-thm}
\zeta(s,x) = \frac{1}{s-1} - \Psi(x) + O(s-1),
\end{equation}
where $\Psi$ is the Digamma function, which is the logarithmic derivative of the gamma function.
\end{theorem}

Thus, $\sum_{n=1}^\infty {f(n)}/{n}$ exists whenever  $\sum_{a=1}^q f(a) = 0$, which we will assume henceforth. This makes $L(s,f)$ an entire function. Moreover, $L(1,f)$ can be expressed as a combination of values of the Digamma function. Using \eqref{hurwitz-thm} we get,
\begin{equation}\label{digamma}
L(1,f) = - \frac{1}{q} \sum_{a=1}^q f(a) \Psi \bigg( \frac{a}{q} \bigg).
\end{equation}

Taking the logarithmic derivative of the identity $\Gamma(x)\Gamma(1-x) = \pi / \sin (\pi x)$, we obtain that 
\begin{equation*}
    \Psi(x) - \Psi(1-x) = \pi \cot (\pi x).
\end{equation*}
Therefore, if $f$ is an odd function, we can pair the terms corresponding to $a$ and $q-a$ in \eqref{digamma} to get that
\begin{equation*}\label{L(1,f)-odd}
    L(1,f)  = \frac{\pi}{q} \sum_{a=1}^{(q-1)/2} f(a) \cot \bigg( \frac{a \pi}{q} \bigg),
\end{equation*}
as $\cot(\pi (1-x)) = - \cot \pi x$, $f(q-a)=-f(a)$ and $\cot(\pi/2)=0$. Similarly, $L(k,f)$ can be evaluated if $k$ and $f$ have the same parity, i.e., if both $k$ and $f$ are either odd or even. In particular,
\begin{equation}\label{L(k,f)-same}
    L(k,f) = - \, \frac{{(-1)}^k}{(k-1)! \, q^k} \, \sum_{a=1}^{(q-1)/2} f(a) {\bigg( \frac{d^{(k-1)}}{dz^{(k-1)}}(\pi \cot \pi z)\bigg\vert_{z=a/q} \bigg)}.
\end{equation}
For a proof of the above fact, we refer the reader to \cite[Theorem 10]{ram-saradha-2}.\\

\subsection{Higher dimensional Dedekind sums}
In the course of evaluating moments of $L(k,f)$, we also encounter a generalization of the higher dimensional Dedekind sums which were introduced by D. Zagier \cite{zagier}. These sums, studied by A. Bayad and A. Raouj \cite{bayad-raouj}, are defined as follows. For $i=0, \cdots, d$, let $a_0$ be a positive integer, $a_1, \cdots, a_d$ be positive integers co-prime to $a_0$ and $m_0, \cdots, m_d$ be non-negative integers. Define
\begin{align}
    & C(a_i \, ; \, a_0, \cdots, \widehat{a_i}, \cdots, a_d \, \vert \, m_i \, ; \, m_0, \cdots, \widehat{m_i}, \cdots, m_d) \nonumber \\
    & := \begin{cases} \label{general-dedekind-sum}
    \frac{1}{a_i^{m_i+1}} \sum_{k=1}^{a_i-1} \prod\limits_{\substack{j=0, \\ j \neq i}}^{d} \bigg( \frac{d^{m_j}}{dz^{m_j}} (\cot z)\bigg)\bigg\vert_{z = \pi k a_j/a_i}, & \text{ if } a_i \geq 2, \\
    0 & \text{ if } a_i=1.  
    \end{cases}
\end{align}
Here $\widehat{x_n}$ means that the term $x_n$ is omitted. It can be shown that the number given by \eqref{general-dedekind-sum} is in fact rational. For a proof of this fact and further properties of the higher dimensional Dedekind sums, see \cite{zagier} and \cite{bayad-raouj}. Moreover, these sums satisfy a reciprocity law given by:
\begin{theorem}\cite[Theorem 2.0.2]{bayad-raouj}\label{reciprocity}
Let $d$ be a positive integer, $a_0, \cdots, a_d$ be pairwise positive integers and $m_0, \cdots, m_d$ be non-negative integers. Let $B_n$ denote the $n^{\text{th}}$ Bernoulli number. Assume that the integer $M = d + m_0 + \cdots + m_d$ is even. Then we have
\begin{align*}
    \sum_{i=0}^d & {(-1)}^{m_i} \, m_i!  \, \sum_{\substack{l_0, \cdots, \widehat{l_i}, \cdots, l_d \geq 0 \\ l_0+\cdots+\widehat{l_i}+\cdots+l_d=m_i}} \bigg( \prod_{\substack{j=0 \\ j \neq i}}^d \frac{a_j^{l_j}}{l_j!}\bigg) \\
    & \times C(a_i \, ; \, a_0, \cdots, \widehat{a_i}, \cdots, a_d \, \vert \, m_i \, ; \, m_0+l_0, \cdots, \widehat{m_i+l_i}, \cdots, m_d+l_d) \\
    & = \begin{cases}
    - \big( R + {(-1)}^{d/2} \big) &\text{ if all } m_i \text{ are zero,}\\
    -  R & \text{ otherwise,}
    \end{cases}
\end{align*}
where \footnote{The minus sign in front of the right hand side is missing in the statement of \cite[Theorem 2.0.2]{bayad-raouj} but is evident from the proof.}
\begin{equation*}
    R = \frac{{(-1)}^{M/2} \, 2^M}{\prod_{i=0}^d a_i^{m_i+1}} \, \, \sum_{\substack{j_0, \cdots, j_d \geq 0 \\ j_0+\cdots+j_d=M/2}} \prod_{i=0}^{d} a_i^{2j_i} A_{i, j_i}
\end{equation*}
and
\begin{equation*}
    A_{i,j_i} =
    \begin{cases}
    \frac{B_{2j_i}}{(2j_i - 1 - m_i)! (2 j_i)} & \text{ if } j_i \text{ is an integer } \geq (m_i+1)/2, \\
    {(-1)}^{m_i} \, m_i! & \text{ if } j_i =0, \\
    0 & \text{ otherwise.}
    \end{cases}
\end{equation*}
\end{theorem}

\subsection{Partial sum lemma} We will use the following elementary exercise from \cite[Problem 70, pg. 16]{polya-szego}.
\begin{lemma}\label{partial-sum-lemma}
Let the sequences $a_n$ and $b_n$ satisfy the conditions:
\begin{equation*}
    b_n > 0, \hspace{3mm} n = 1,2, \cdots \hspace{1mm} ; \hspace{2mm}  b_1 + b_2 + b_3 + \cdots + b_n + \cdots \text{ diverges;}
\end{equation*}
and
\begin{equation*}
    \lim_{n \rightarrow \infty} \frac{a_n}{b_n} = s.
\end{equation*}
Then
\begin{equation*}
    \lim_{n \rightarrow \infty} \frac{a_1 + a_2 + \cdots + a_n}{b_1 + b_2 + \cdots + b_n} = s.
\end{equation*}
\end{lemma}

\section{Proof of Theorems}

\subsection{Proof of Theorem \ref{distribution}}
Our proof is based on the following lemma from \cite[Theorem 30.2, pg. 390]{billingsley}.
\begin{lemma}\label{main-lemma}
Suppose the distribution of $X$ is determined by its moments, that the $X_n$ have moments of all orders and that $\lim_n E[ X_n^r] = E[X^r]$ for $r=1,2, \cdots$. Then $X_n \Rightarrow X$, i.e., the distribution of $X_n$ converges to the distribution of $X$ wherever the distribution of $X$ is continuous.
\end{lemma}

Thus, we first show that
\begin{equation*}
    M(n) := \lim_{r \rightarrow \infty} m_{2r+1}(n) = \lim_{r \rightarrow \infty} \frac{\sum_{f \in E_{2r+1}^{(k)}}L(k,f)^n}{\# E_{2r+1}^{(k)}}
\end{equation*}
is finite. In order to show that the limiting distribution is determined by its moments, we use a general theorem from \cite[Theorem 30.1, pg. 388]{billingsley} stated below.
\begin{theorem}\label{radius-cvg}
Let $\mu$ be a probability measure on the line having finite moments $\alpha_k = \int_{- \infty}^{\infty} x^k \mu(dx)$ of all orders. If the power series $\sum_{k=0}^{\infty} \alpha_k \, r^k/k!$ has a positive radius of convergence, then $\mu$ is determined by its moments.
\end{theorem}

To begin with, we evaluate $$m_q(n):= \frac{\sum_{f \in E_q^{(k)}}L(k,f)^n}{\# E_q^{(k)}},$$ for any non-negative integer $n$ in terms of sums of the form \eqref{general-dedekind-sum}. By \eqref{L(k,f)-same},
\begin{align*}
   & \sum_{f \in E_{q}^{(k)}} {L(k,f)}^n \\
   & = {\frac{{(-1)}^{kn}}{{\big((k-1)!\big)}^n \,  q^{kn}}}\, \sum_{a_1, \cdots, a_n =1}^{(q-1)/2} \prod_{j=1}^n {\bigg( \frac{d^{(k-1)}}{dz^{(k-1)}}(\pi \cot \pi z)\bigg\vert_{z=a_j/q} \bigg)} \sum_{f \in E_q^{(k)}} f(a_1) \cdots f(a_n), \\
    & = {\frac{{(-1)}^{kn} \, \pi^k}{{\big((k-1)!\big)}^n \,  q^{kn}}}\, \sum_{a_1, \cdots, a_n=1}^{(q-1)/2} \prod_{j=1}^n {\bigg( \frac{d^{(k-1)}}{dz^{(k-1)}}( \cot z)\bigg\vert_{z=a_j/q} \bigg)} \sum_{f \in E_q^{(k)}} f(a_1) \cdots f(a_n).
\end{align*}

Note that if $f \in E_q^{(k)}$ then $-f \in E_q^{(k)}$. Thus, if $n$ is odd, the inner sum becomes zero by pairing terms corresponding to $f$ and $-f$. Therefore, $m_q(n)=0$ when $n$ is odd.

Henceforth, let $n$ be even and $p(n)$ denote the number of partitions of $n$. The above sum can be partitioned into $p(n)$ many sums according to the equality of the indices $a_1, \cdots, a_n$. In particular, for a partition $\underline{\lambda}=(\lambda_1, \cdots, \lambda_m)$ of $n$, the inner sum becomes
\begin{equation}\label{inner-partition-sum}
    \sum_{f \in E_q^{(k)}} f(a_1)^{\lambda_1} \cdots f(a_m)^{\lambda_m},
\end{equation}
where $1 \leq a_1, \cdots, a_m \leq q-1$ are all distinct. Clearly, the above sum is $\# E_q^{(k)}$ when $\lambda_l$ is even for all $1 \leq l \leq m$. Now, without loss of generality, suppose that $\lambda_1$ is odd. Since even and odd functions are determined by their values on $1 \leq a \leq r$, for any $f \in E_q^{(k)},$ there is a unique $f^{-} \in E_q^{(k)}$ such that \begin{equation*}
    f^{-}(n) =
    \begin{cases}
    f(n) & \text{ if } n \neq a_1, \\
    -f(n) & \text{ if } n = a_1,
    \end{cases}
\end{equation*} 
for $1 \leq n \leq r$. Pairing up the terms corresponding to $f$ and $f^{-}$ in \eqref{inner-partition-sum}, we obtain that
\begin{equation*}
    \sum_{f \in E_q^{(k)}} f(a_1)^{\lambda_1} \cdots f(a_m)^{\lambda_m} = 
    \begin{cases}
    \# E_q^{(k)}, & \text{ if } \lambda_1, \cdots, \lambda_m \text{ are even},\\
    0 & \text{ otherwise.}
    \end{cases}
\end{equation*}

Thus, only those terms corresponding to partitions consisting of even parts survive and one can write 
\begin{align*}
    & \sum_{f \in E_{q}^{(k)}} {L(k,f)}^n \\
    & =  \frac{ \pi^k}{{\big((k-1)!\big)}^{2n} \, 2^{2n} \, q^{2kn}} \, \sum_{\substack{\underline{\lambda}=(\lambda_1, \cdots, \lambda_m),\\ \underline{\lambda} \in P_n}} \, \sideset{}{'}\sum_{a_1, \cdots, a_m=1}^{q-1} \, \prod_{j=1}^m  {\bigg( \frac{d^{(k-1)}}{dz^{(k-1)}}( \cot z)\bigg\vert_{z=a_j/q} \bigg)}^{2 \lambda_j},
\end{align*}
where $\sideset{}{'}\sum$ denotes that the sum is taken over distinct $1 \leq a_1, \cdots, a_m \leq q-1$. The inner sum can be expressed in terms of generalized higher dimensional Dedekind sums as follows. For any positive integer $u$, define
\begin{equation*}
    S_{q,k}^{(u)} := q \,\, C(q \, ; \, \underbrace{1,1, \cdots,1}_{2u \text{ times}} \, \vert \, 0 \, ; \, \underbrace{k-1, \cdots, k-1}_{2u \text{ times}}).
\end{equation*}
With the notation as in Section \ref{intro}, define $\mathfrak{S}_{q,k}^{(\underline{\eta})}$ inductively as follows. $\mathfrak{S}_{q,k}^{(n)}:= {S_{q,k}^{(n)}}$ and for any $\underline{\lambda} = (\lambda_1, \cdots, \lambda_m)$,
\begin{equation}\label{script-S-def}
    \mathfrak{S}_{q,k}^{(\underline{\lambda})} := S_{q,k}^{(\lambda_1)} \, \cdots \, S_{q,k}^{(\lambda_m)} - \sum_{\underline{\eta} \leq \underline{\lambda}} {\underline{\lambda} \choose \underline{\eta}} \, \, \mathfrak{S}_{q,k}^{(\underline{\eta})}.
\end{equation}
Thus, we have that
\begin{equation*}
    m_q(2n) = \frac{\pi^k}{{\big((k-1)!\big)}^{2n} \, 2^{2n} \, q^{2kn}}  \hspace{1mm} \bigg\{ \sum_{\underline{\lambda} \in P_n} \mathfrak{S}_{q,k}^{\underline{\lambda}} \bigg\}.
\end{equation*}

To understand the asymptotic behaviour of $m_q(2n)$ as $q \rightarrow \infty$, we use the explicit evaluation of $S_{q,k}^{(u)}$ using Theorem \ref{reciprocity}. Thus,
\begin{align*}
    S_{q,k}^{(u)} & = \sum_{t=1}^{q-1}  {\bigg( \frac{d^{k-1}}{dz^{k-1}} (\cot z)\bigg)\bigg\vert}^{2u}_{z = \pi t/q}\\
    &=
    \begin{cases}
    - q \, \big(R + {(-1)}^u \big) & \text{ if } k=1, \\
    - q \, R & \text{ otherwise,}
    \end{cases}
\end{align*}
where 
\begin{equation*}
    R = \frac{{(-1)}^{uk} \, 2^{2uk}}{q} \sum_{\substack{j_0, \cdots, j_{2u}\\ j_0+\cdots+j_{2u}=uk}} \alpha_{j_0,\cdots,j_{2u}} \, q^{2j_0},
\end{equation*}
for $\alpha_{j_0,\cdots,j_{2u}} \in \mathbb{Q}$ given by Theorem \ref{reciprocity}. In particular, $S_{q,k}^{(u)}$ is a polynomial in $q$ of degree $2uk$ with leading coefficient $-{(-4)}^{uk} \alpha_{uk,0,\cdots,0}$, given explicitly by Theorem \ref{reciprocity}. Thus,
\begin{equation*}\label{asymptotic}
    S_{q,k}^{(u)} \sim \bigg( {2}^{2uk} \, {\big( (k-1)! \big)}^{2u} \, {(-1)}^{uk+1} \, \frac{B_{2uk}}{(2uk)!} \bigg) \, q^{2uk}, \hspace{2mm} \text{ as } q \rightarrow \infty.
\end{equation*}

Using this, \eqref{script-S-def} and the definition of $c(\underline{\lambda})$ \eqref{c(lambda)}, we get that
\begin{equation*}
    \mathfrak{S}_{q,k}^{(\underline{\lambda})} \sim  c(\underline{\lambda}) \, q^{2nk},  
\end{equation*}
as $q$ tends to infinity. Hence,
\begin{equation*}
   M(2n) = \lim_{q \rightarrow \infty} m_q(2n) = \frac{\pi^{2nk}}{{((k-1)!)}^{2n} \, 2^{2n}} \, \bigg(\sum_{\underline{\lambda} \in P_n} c(\underline{\lambda}) \bigg).
\end{equation*}

Since the limit as $q$ tends to infinity of $m_q(2n)$ exists, by Lemma \ref{main-lemma}, there exists a limiting distribution $F(x)$ whose odd moments are zero and even moments are given by $M(2n)$. Thus, the characteristic function of $F(x)$ is given by
\begin{equation*}
    \phi(t) := \sum_{n=0}^{\infty} \frac{M(2n)}{(2n)!} \, {(it)}^{2n}.
\end{equation*}
By Lemma \ref{main-lemma} and Lemma \ref{radius-cvg}, it suffices to show that $\phi(t)$ has positive radius of convergence. In fact, we prove that it is entire.\\ 

Let $\zeta(s) = \sum_{n=1}^{\infty} 1/n^s$ for $\Re(s)>1$ be the Riemann zeta function. In 1737, Euler showed that
\begin{equation*}
    \zeta(2m) = {(-1)}^{m+1} \, \frac{B_{2m} \, {(2 \pi)}^{2m}}{(2m)! \, 2}.
\end{equation*}

Since $\zeta(2m+2) < \zeta(2m)$, 
\begin{equation*}
     \bigg| \frac{B_{2m+2} }{(2m+2)!} \bigg| < \bigg| \frac{B_{2m+2} \, \pi^2}{(2m+2)!} \bigg| < \bigg| \frac{B_{2m} }{(2m)!} \bigg|.
\end{equation*}
Hence,
\begin{equation*}
    c(\underline{\lambda}) \leq \bigg| 2^{2nk} {((k-1)!)}^{2n} \bigg(\prod_{i=1}^m {(-1)}^{\lambda_{i}k+1} \frac{B_{2\lambda_{i}k}}{(2\lambda_{i}k)!} \bigg) \bigg| \leq 2^{2nk} {((k-1)!)}^{2n} {\bigg| \frac{B_{2k}}{(2k)!} \bigg|}^m.
\end{equation*}
and thus,
\begin{equation*}
    M(2n) = \frac{\pi^{2nk}}{{((k-1)!)}^{2n} \, 2^{2n}} \, \bigg(\sum_{\underline{\lambda} \in P_n} c(\underline{\lambda}) \bigg) \leq \frac{{(2 \pi)}^{2nk}}{2^{2n}} \, p(n) \, {\bigg| \frac{B_{2k}}{(2k)!} \bigg|}^n = \frac{p(n)}{2^n} \, {\bigg(\frac{\zeta(2k)}{2} \bigg)}^n,
\end{equation*}
where $p(n)$ denotes the number of partitions of $n$. In 1918, Hardy and Ramanujan \cite{hardy-ram} showed that
\begin{equation*}
    p(n) \sim \frac{1}{(4\sqrt{3}) \, n} \, \exp \bigg( \pi \sqrt{\frac{2n}{3}}\bigg) \text{    as } n \rightarrow \infty.
\end{equation*}

Now using Stirling's formula and the asymptotics for $p(n)$, we obtain that
\begin{equation*}
    {\bigg( \frac{M(2n)}{(2n)!} \bigg)}^{1/2n} \ll e^{-c \log n},
\end{equation*}
for a positive constant $c$. Therefore, applying the root test gives that the radius of convergence of $\phi(t)$ is infinite. This proves the theorem.

\subsection{Proof of Theorem \ref{density-1}}

Let $q \geq 3$ be odd and $E_q$ be the set of all Erd\H{o}s functions mod $q$. Let $r := (q-1)/2$. We define a relation on this set as follows. For $f,g \in E_q$,
\begin{equation*}
    f \sim g \iff f(a) = g(a), \hspace{1mm} \forall 1 \leq a \leq q, \hspace{1mm} (a,q) \neq 1.
\end{equation*}
One can easily check that $\sim$ is an equivalence relation. Before proceeding, we prove the following proposition.
\begin{proposition}\label{L(1,f)-prop}
There exists at most one Erd\H{o}s function $f$ in every equivalence class of $E_q$ under $\sim$, such that $L(1,f)=0$.
\end{proposition}
\begin{proof}
Suppose $f, g \in E_q$ are such that $L(1,f)=L(1,g)=0$. Thus, by \eqref{digamma} and the convergence condition, we have
\begin{align*}
    \sum_{\substack{a=1 \\ (a,q)=1}}^q f(a) \bigg[ \Psi \bigg( \frac{a}{q} \bigg) + \gamma \bigg] & = - \, \sum_{\substack{a=1\\ (a,q) \neq 1}}^q f(a) \bigg[ \Psi \bigg( \frac{a}{q} \bigg) + \gamma \bigg] \\
    & = - \, \sum_{\substack{a=1 \\ (a,q) \neq 1}}^q g(a) \bigg[ \Psi \bigg( \frac{a}{q} \bigg) + \gamma \bigg] \\
    & = \sum_{\substack{a=1 \\ (a,q)=1}}^q g(a) \bigg[ \Psi \bigg( \frac{a}{q} \bigg) + \gamma \bigg].
\end{align*}
Therefore, we obtain 
\begin{equation*}
    \sum_{\substack{a=1 \\ (a,q)=1}}^q \big[f(a)-g(a)\big] \, \bigg[ \Psi \bigg( \frac{a}{q} \bigg) + \gamma \bigg],
\end{equation*}
which is a $\mathbb{Q}$-linear relation among the numbers
\begin{equation*}
    \Psi(a/q) + \gamma, \hspace{2mm} 1 \leq a \leq q, \hspace{1mm} (a,q)=1.
\end{equation*}
But these numbers are $\mathbb{Q}$-linearly independent as proven in \cite[Theorem 4]{ram-saradha}. Hence, $f=g$. 
\end{proof}

Therefore, it suffices to count the number of equivalence classes of $E_q$ under $\sim$. In order to count these, note that each equivalence class differs from the other based on the values of functions on $N_q := \{ \, a \, : \, 1 \leq a < q, \, (a,q) \neq 1\}$. At each point in this set, an Erd\H{o}s function can take either $1$ or $-1$, with the only restriction that 
\begin{equation*}
    \# \{ a \in N_q \, : \, f(a) = 1\} = \# \{ a \in N_q \, : \, f(a) = -1\} = \frac{q-1-\phi(q)}{2}.
\end{equation*}
Owing to this, two cases arise. For simplicity of notation, let $n_q:= (q-1-\phi(q))/2$ and recall that $r = (q-1)/2$. \\

\begin{enumerate}[(a)]
    \item $r \geq n_q$: In this case, the number of $a \in N_q$ where a function takes the value $1$ ranges from $0$ to $n_q$. Thus, the total number of equivalence classes is 
    \begin{equation*}
        \big| E_q / \sim \big| = \sum_{k=0}^{n_q} {n_q \choose k} = 2^{n_q}.
    \end{equation*}
    
    \item $r < n_q$: Let $j := n_q - r$. Then, the number of $a \in N_q$ where a function takes the value $1$ has to be at least $j$. Hence, the number of equivalence classes is
    \begin{equation*}
        \big| E_q / \sim \big| = \sum_{k=j}^{n_q} {n_q \choose k} < 2^{n_q}.
    \end{equation*}
\end{enumerate}

Therefore, in either case $| V_q | \leq 2^{(q-1-\phi(q))}$. Now, note that 
\begin{equation*}
    \big| E_q \big| = {2r \choose r}. 
\end{equation*}
Using the bounds by \cite{robbins}, one has
\begin{equation*}
    \sqrt{2 \pi} \, n^{(n + \frac{1}{2})} \, e^{-n}  < n! <  \sqrt{2 \pi} \, e \, n^{(n + \frac{1}{2})} \, e^{-n},
\end{equation*}
for all $n \in \mathbb{N}$, we get that
\begin{equation*}
    {2r \choose r} = \frac{(2r) \, !}{{(r!)}^2} \geq \frac{\sqrt{2 \pi} \, {(2r)}^{(2r + 1/2)} \, e^{-2r} }{{(\sqrt{2 \pi} \, e \, r^{(r + 1/2)} \, e^{-r})}^2} = \frac{\sqrt{2} \, 2^{2r} \, r^{2r} \, \sqrt{r}}{\sqrt{2 \pi} \, e^2 \, r^{2r} \, r } = \frac{2^{2r}}{e^2\sqrt{\pi} \, \sqrt{r}}.
\end{equation*}
Thus, 
\begin{equation*}\label{less than}
     \bigg( {\sum_{\substack{3 \leq q \leq x, \\ q \text{ odd }}} \# V_q} \bigg) \bigg/ \bigg( {\sum_{\substack{3 \leq q \leq x,\\ q \text{ odd }}} \# E_q } \bigg)  \ll  \frac{\sum_{r \leq x} 2^{2r-\phi(2r+1)}}{\sum_{r \leq x}( 2^{2r}/\sqrt{r})} \ll \frac{\sum_{r \leq x} 2^{2r-(2r/(\log \log r))}}{\sum_{r \leq x}( 2^{2r}/\sqrt{r})}
\end{equation*}
because by \cite[pg. 217]{landau},
\begin{equation*}
    \liminf_{n \rightarrow \infty} \frac{\phi(n) \log \log n}{n} = e^{-C}. 
\end{equation*}
Therefore, by Lemma \ref{partial-sum-lemma}, the right hand side tends to zero as $x \rightarrow \infty$. This proves Theorem \ref{density-1}.

\section{Conclusion}

Theorem \ref{distribution} leads to various natural questions regarding the distribution function $F(x)$. For example, the continuity of $F$, the rate of convergence in distribution as well as the tail of the distribution are but a few problems arising out of our study. Another direction is to understand the distribution of $L(k,f)$ without any restriction on the parity of Erd\H{o}s functions. In this course, one is led to study the analogue of higher dimensional Dedekind sums for the digamma function. Since these investigations are far afield from our current focus, we relegate them to future research.

\section*{Acknowledgements}
I extend my sincere gratitude to Prof. M. Ram Murty for drawing my attention to this question and for a careful reading of this article. I am also thankful to Abhishek Bharadwaj, Anup Dixit, Tapas Chatterjee and the referee for insightful comments on an earlier version of this paper.

\end{document}